\documentclass[a4paper,12pt,notitlepage]{article}
\usepackage[utf8]{inputenc}

\usepackage{amsmath}
\usepackage{amsthm}
\usepackage{amsfonts}

 \usepackage{enumerate}
\usepackage{graphicx}
\usepackage{verbatim}
\usepackage{hyperref}

 \newcommand{\eps}{\varepsilon}
 
  \newcommand{\bN}{\mathbb{N}}
  \newcommand{\bR}{\mathbb{R}}
 
 \newcommand{\cA}{\mathcal{A}}

 \newcommand{\cP}{\mathcal{P}}

 \newcommand{\Ds}{D^{\ast}}
 \newcommand{\cs}{c^{\ast}}
 \newcommand{\fs}{f^{\ast}_{\text{strong}}}
 \newcommand{\fss}{f^{\ast}_{\text{strict}}}
 \newcommand{\Gs}{\Gamma^{\ast}}
 \newcommand{\xis}{\xi^{\ast}}

 \newcommand{\I}{[0,1)}

 \newcommand{\ox}{\overline{x}}
   \newcommand{\on}{\overline{n}}

  \newcommand{\ux}{\underline{x}}
  \newcommand{\un}{\underline{n}}

\newcommand{\Qp}{Q^{\prime}}
  \newcommand{\Qpp}{Q^{\prime\prime}}
 
  \newcommand{\tA}{\tilde{A}}
  \newcommand{\txi}{\tilde{\xi}}

 \newcommand{\tf}{\tilde{f}}
 \newcommand{\tG}{\tilde{\Gamma}}

 \newcommand{\cc}{\chi_{\text{crit}}}
  \newcommand{\cm}{\chi_{\text{min}}}
 
 \newcommand{\ud}{\, \mathrm{d}}

\theoremstyle{plain}		\newtheorem{Theorem}{Theorem}
				
				\newtheorem{Lemma}{Lemma}
\theoremstyle{definition}	
				
				\newtheorem{Definition}{Definition}

\theoremstyle{plain}		\newtheorem*{Theorem*}{Theorem}
				\newtheorem*{Property*}{Property}
\theoremstyle{definition}	\newtheorem*{Proposition*}{Proposition}
				\newtheorem*{Remark*}{Remark}
				\newtheorem*{Definition*}{Definition}				

 \title{An improved bound for the star discrepancy of sequences in the unit interval}
 \date{}
 
\author{Gerhard Larcher\thanks{The authors are supported by the Austrian Science Fund (FWF),  Project F5507-N26, which is a part of the Special Research Program ``Quasi-Monte Carlo Methods: Theory and Applications''}  \and 
  Florian Puchhammer\footnotemark[1]}

\begin{document}

\maketitle

 \begin{abstract}
It is known that there is a constant $c>0$ such that for every sequence $x_1, x_2,\ldots$ in $\I$ we have for the star discrepancy $\Ds_N$ of the first $N$ elements of the sequence that $N\Ds_N\geq c\cdot \log N$ holds for infinitely many $N$. Let $\cs$ be the supremum of all such $c$ with this property. We show $\cs>0.065664679\ldots$, thereby slightly improving the estimates known until now.
 \end{abstract}

\section{Introduction and statement of the result}
\label{sec:intro}
Let $x_1, x_2,\ldots$ be a point sequence in $\I$. By $\Ds_N$ we denote the star discrepancy of the first $N$ elements of the sequence, i.e., 
\begin{equation*}
 \Ds_N=\sup_{x\in[0,1]}\left|\frac{\cA_N(x)}{N}-x\right|,
\end{equation*}
where $\cA_N(x):=\#\{1\leq n\leq N|x_n<x\}$. The sequence $x_1,x_2,\ldots$ is uniformly distributed in $\I$ iff $\lim_{N\to\infty}\Ds_N=0$.

In 1972 W.~M.~Schmidt \cite{SchIrr72} showed that there is a positive constant $c$ such that for all sequences $x_1, x_2,\ldots$ in $\I$ we have
\begin{equation}
\label{eqn:dsschmidt}
 \Ds_N>c\cdot\frac{\log N}{N}
\end{equation}
for infinitely many $N$. The order $\frac{\log N}{N}$ is best possible. There are many  known sequences for which $\Ds_N\leq c^{\prime}\cdot\frac{\log N}{N}$ holds for all $N$ with an absolute constant $c^{\prime}$.
For all necessary details on discrepancy and low-discrepancy sequences see the monographs \cite{KuiUni74} or \cite{NieRan92}.

So it makes sense to define the \emph{one-dimensional star discrepancy constant} $\cs$ to be the supremum over all $c$ such that (\ref{eqn:dsschmidt}) holds for all sequences $x_1, x_2,\ldots$ in $\I$ for infinitely many $N$. Or, in other words,
\begin{equation*}
 \cs:=\inf_{\omega}\limsup_{N\to\infty}\frac{N\Ds_N(\omega)}{\log N},
\end{equation*}
where the infimum is taken over all sequences $\omega=x_1,x_2,\ldots$ in $\I$, and where $\Ds_N(\omega)$ denotes the star discrepancy of the first $N$ elements of $\omega$.

The currently best known estimates for $\cs$ are
\begin{equation*}
 0.0646363\ldots\leq\cs\leq0.222\ldots
\end{equation*}
The upper bound was given by Ostromoukhov \cite{OstRec08} (thereby slightly improving earlier results of Faure (see, for example, \cite{FauGoo92})). The lower bound was given by Larcher \cite{LarOnt14}.

It is the aim of this paper to improve the above lower bound for $\cs$. That is, we prove
\begin{Theorem}
\label{thm:main}
For the one-dimensional star discrepancy constant we have
\begin{equation*}
 \cs\geq0.065664679\ldots
 \end{equation*}
\end{Theorem}
\bigskip
The idea of the proof follows a method introduced by Liardet \cite{LiaDis79} which was also used by Tijdeman and Wagner in \cite{TijSeq80} and by Larcher in \cite{LarOnt14}.

\section{Main ideas and proof of Theorem~\ref{thm:main}}
\label{sec:idea}

We will heavily make use of the idea, the notation, and most of the results used and obtained in \cite{LarOnt14}. In this paper we extend the analysis carried out in the aforementioned paper. In this section we therefore repeat the most important notation and facts from \cite{LarOnt14} and explain how we extend the method to prove Theorem~\ref{thm:main}.

We consider a finite point set $\cP=\{x_1, x_2,\ldots,x_N\}$ in $\I$ with $N=[a^{t}]$ for some real $a$, $3\leq a\leq3.7$, and some $t\in\bN$. Further, we divide the index-set $A=\{1,2,\ldots,N\}$ into index-subsets $A_0, A_1, A_2$, where $A_0=\{1,2,\ldots,[a^{t-1}]\}$, $A_2=\{[a^{t}]-[a^{t-1}]+1,\ldots,[a^{t}]\}$, and $A_1=A\setminus(A_0\cup A_2)$.

For simplicity, let us first of all assume that $a^{t}$ and $a^{t-1}$ are integers (of course this can only happen if $a=3$). For $x\in\I$ we consider the discrepancy function $D_n(x):=\#\{i\leq n|x_i<x\}-nx=\cA_n(x)-nx$ and we define the function $f(x):=\max_{n\in A_2}D_n(x)-\max_{n\in A_0}D_n(x)$.

In \cite{LarOnt14} it was shown that the function $f$ has the following properties:
\begin{enumerate}[(i)]
 \item $f(0)=f(1)=0$.
 \item $f$ is piecewise linear, piecewise monotonically decreasing, and $|f|$ is bounded by $a^{t}$.
 \item $f$ is left-continuous and each discontinuity constitutes a positive jump.
 \item The slope of $f$ is always between $-a^{t}$ and $s_0:=-a^{t-1}(a-2)$.
 \item If $f$ is continuous on $[x,y]$ then the slope of $f(x)$ and $f(y)$ can differ at most by $a^{t-1}$.
 \item $f$ has discontinuities with a jump of height at least 1 in all points $x_i$ with $i\in A_1$.
\end{enumerate}

Further it was shown in \cite[Lemma 2.11]{LarOnt14} that for given $a$ and $t$ there exists a function $\fs\!:~[0,1]\rightarrow\bR$ satisfying (i)--(vi) for some $x_1,\ldots,x_N$ (we say $\fs$ is \emph{strongly admissible}) such that
\begin{equation*}
 \int_{0}^{1}\left|\fs(x)\right|\ud x=\min_{g \text{ strongly admissible}} \int_{0}^{1}\left|g(x)\right|\ud x,
\end{equation*}
and (in \cite[Lemma 2.14]{LarOnt14}) that for every $\eps>0$ and (now arbitrary) $a\in[3,4]$ and $t$ with $t\geq t(\eps)$
\begin{equation*}
 \int_{0}^{1}|\fs(x)|\ud x \geq \frac{(a-2)(8a+3)}{8(1-2a)^2}-\eps.
\end{equation*}
Finally, we finished the proof of the Theorem in \cite{LarOnt14} in the following way:

It was shown that (see Section 3 in \cite{LarOnt14})
\begin{align*}
 \int_{0}^{1}\left( \max_{n\in A}D_n(x)-\right.&\left.\min_{n\in A}D_n(x)  \right)\ud x 
 \geq t\int^{1}_{0}\left|\fs(x)\right|\ud x
 \geq\\
 &\geq t\left( \frac{(a-2)(8a+3)}{8(1-2a)^2}-\eps \right)
  \geq\\
 &\geq \frac{\log N}{\log a}\cdot\left(\frac{(a-2)(8a+3)}{8(1-2a)^2}-\eps\right)
 \geq\\
 &\geq2\log N\cdot0.0646363\ldots
\end{align*}
if we choose $a=3.71866\ldots$ and $N$ large enough. Hence there exist $x\in[0,1]$ and $n\leq N$ with
\begin{equation*}
 D_n(x)\geq 0.0646363\ldots\cdot\log N
 \end{equation*}
and Theorem 1.1 from \cite{LarOnt14} follows.

To improve the above result from \cite{LarOnt14} in the present paper we proceed as follows: We show that $f$ has to satisfy an even more restrictive property (vi$^{\prime}$) instead of property (vi) and we call a function $g$ satisfying (i)--(v) and (vi$^{\prime}$) \emph{strictly admissible}. Moreover, we show that there exists a strictly admissible function $\fss\!:~[0,1]\rightarrow\bR$ with
\begin{equation*}
 \int_{0}^{1}\left|\fss(x)\right|\ud x=\min_{g\text{ strictly admissible}}\int_{0}^{1}\left|g(x)\right|\ud x
\end{equation*}
and
\begin{equation*}
 \int_{0}^{1}\left|\fss(x)\right|\ud x\geq
\frac{(a-2)\left(12a+9+(a-2)(4a-3)\log\left(1+\frac{1}{a-2}\right)\right)}{a\left(a-\frac{1}{2}\right)^2\left(3+(a-2)\log\left(1+\frac{1}{a-2}\right)\right)}-\eps
\end{equation*}
for all $a\in(3,3.7]$ and $t\geq t(\eps)$.

Note that, in the following, we will work with $a^{t}$ and $a^{t-1}$ as if they were integers and we will obtain the above result without ``$-\eps$'' and for all $t\geq t_0$ in this case. For working with $[a^{t-1}]$ and $[a^t]$ instead of $a^{t-1}$ and $a^{t}$ we then easily obtain the stated result.

In the very same way as in \cite{LarOnt14} and as described above we then obtain $D_n(x)\geq0.065664679\ldots\cdot\log N$ for some $x\in[0,1]$ and $n\geq N$ by choosing $a=3.62079\ldots$. Consequently, Theorem~\ref{thm:main} follows.

So it remains to prove the two main auxiliary results, namely, that a stronger property (vi$^{\prime}$) for $f$ as well as the lower bound for $\int_{0}^{1}|\fss(x)|\ud x$ as stated above hold. This is carried out in the next section. For the proofs of these two results we will have to use some facts already obtained  in \cite{LarOnt14}, again.

\section{Proof of the auxiliary results}
\label{sec:auxiliary}

\begin{Lemma}
\label{lemma:bendcond}
 Let $j\in A_2$, i.e., $j=a^{t}-a^{t-1}+k$ for some integer $k$, $1\leq k<a^{t-1}$, and assume that $f(x)=\max_{n\in A_2}D_n(x)-\min_{n\in A_0}D_n(x)$ has a discontinuity in $x_j$. Let further $l_j$, $r_j\in A$ such that $\cP\cap(x_{l_j},x_{r_j})=\{x_j\}$. If there exists an $\ox\in(x_j,x_{r_j})$ such that, in $\ox$ $f$ has slope $s(\ox)>s_0-k$ then $f(\ux)\geq f(\ox)-s_0(\ox-\ux)$ for all $\ux\in[x_{l_j},x_j)$. Here, again,  $s_0 = a^{t-1} (a-2)$ as defined in property (iv) above.
\end{Lemma}

\begin{Remark*}
 The meaning of Lemma~\ref{lemma:bendcond} is illustrated in Figure~\ref{fig:bendcond}. Using the same notation  $f(\ux)$ lies above the line with slope $s_0$ reaching back from the point  $(\ox,f(\ox))$ (dashed)  in case the slope of $f$ (solid) becomes flatter than $s_0-k$. 
\end{Remark*}

\begin{figure}[!htbp]
 \centering
 \includegraphics[width=0.5\textwidth]{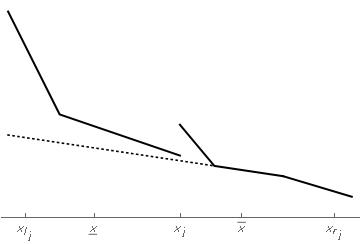}
 \caption{}
 \label{fig:bendcond}
\end{figure}

\begin{proof}[Proof of Lemma~\ref{lemma:bendcond}.]
 Let $\ux$, $\ox$ be like above with $s(\ox)>s_0-k$. We set $\on_i=n_i(\ox)$ and $\un_i=n_i(\ux)$ such that $D_{\on_i}(\ox)=\max_{n\in A_i}D_n(\ox)$ and $D_{\un_i}(\ux)=\max_{n\in A_i}D_n(\ux)$. So $f(\ox)=D_{\on_2}(\ox)-D_{\on_0}(\ox)$ and $f(\ux)=D_{\un_2}(\ux)-D_{\un_0}(\ux)$.
 
 First we show that $\on_2<j$. Indeed, we have
 \begin{equation*}
  a^{t-1}-\on_2\geq\on_0-\on_2=s(\ox)>s_0-k=-a^{t-1}(a-2)-k.
  \end{equation*}
Thus, $\on_2<a^{t}-a^{t-1}+k=j$.

Since $\cA_n$ does not change its value in $x_j$ for $n<j$, $D_{\on_2}$ does not have a jump in $x_j$.  Consequently, $D_{\on_2}(\ox)=D_{\on_2}(\ux)-\on_2(\ox-\ux)$. This observation yields
\begin{equation*}
 D_{\un_2}(\ux)-D_{\on_2}(\ox)\geq D_{\on_2}(\ux)-D_{\on_2}(\ox)=\on_2(\ox-\ux).
\end{equation*}
By the same argument we additionally obtain
\begin{equation*}
 D_{\un_0}(\ux)-D_{\on_0}(\ox)\leq D_{\un_0}(\ux)-D_{\un_0}(\ox)=\un_0(\ox-\ux).
 \end{equation*}
Alltogether
\begin{align*}
 f(\ux)-f(\ox)&=\left( D_{\un_2}(\ux)-D_{\on_2}(\ox) \right) - \left( D_{\un_0}(\ux)-D_{\on_0}(\ox) \right)\\
 &\geq(\on_2-\un_0)(\ox-\ux)\geq-s_0(\ox-\ux)
\end{align*}
and the result follows.
\end{proof}

In addition to the new property of $f$ obtained in Lemma~\ref{lemma:bendcond} one can easily convince oneself that $f$ is continuous at $x_1$. This result is not very effectful yet nice for calculation purposes. We will use this fact in the following concept of strict admissibility.

\begin{Definition}
 \label{def:strictadm}
 A function $g\!:[0,1]\rightarrow\bR$ is called strictly admissible if it satisfies conditions (i)--(v) and the following additional condition (vi$^{\prime}$):
 
 There exists a set $\Gamma=\{\xi_1,\xi_2,\ldots,\xi_{a^{t}-1}\}\subset\I$ such that:
 \begin{enumerate}[a)]
  \item If $g$ has a jump in $\xi$ then $\xi\in\Gamma$.
   \item There exists a set $\Gamma_1\subset\Gamma$, $|\Gamma_1|=a^{t-1}(a-2)$, such that  $f$ has a jump of height at least one in each $\xi\in\Gamma_1$.
   \item There exist $a^{t-1}-1$ further points $\{\xi_{k_1},\xi_{k_2},\ldots,\xi_{k_{a^{t-1}-1}}\}=:\Gamma_2$ with the following property:
   
   For each $1\leq n<a^{t-1}$ let $\xi_{l_n},\xi_{r_n}\in\Gamma\cup\{0,1\}$ such that $\Gamma\cap(\xi_{l_n},\xi_{r_n})=\{\xi_{k_n}\}$. Now, if there is an $\ox\in(\xi_{l_k},\xi_{r_k})$ with
 \begin{equation}
  \label{eqn:nr1}
  s(\ox)>s_0-n
 \end{equation}
then 
\begin{equation}
 \label{eqn:nr2}
 g(\ux)\geq g(\ox)-s_0(\ox-\ux)
 \end{equation}
for all $\ux\in[\xi_{l_n},\xi_{k_n})$. Here, $s(x)$ denotes the slope of $g$ in $x$.
 \end{enumerate}
\end{Definition}

From \cite{LarOnt14} and from Lemma~\ref{lemma:bendcond} it follows that $f$ is strictly admissible. The space of strictly admissible functions, again, is obviously closed with respect to pointwise convergence. Hence, there exists $\fss$ strictly admissible with
\begin{equation*}
 \int_{0}^{1}|f(x)|\ud x\geq \min_{g\text{ strictly admissible}} \int_{0}^{1}|g(x)|\ud x=\int_{0}^{1}|\fss(x)|\ud x.
\end{equation*}
We intend to estimate $\int_{0}^{1}|\fss(x)|\ud x$ from below. To this end we have to derive some properties of $\fss$.

\begin{Lemma}
 \label{lemma:zeros}
 Let $\fss$ have a discontinuity in $\gamma$. Then there exist two zeros $\alpha$, $\beta$ of $\fss$ with $\alpha<\gamma<\beta$ such that $\gamma$ is the only discontinuity in the interval $(\alpha,\beta).$
\end{Lemma}

\begin{proof}
First of all, if $\gamma$ is the only point at which $\fss$ has a jump, the claim is fulfilled with $\alpha=0$ and $\beta=1$. Hence it suffices to show the following statement: Let $\fss$ have two successive discontinuities in, say, $a_1$ and $a_2$, $0<a_1<a_2<1$. Then $\fss$ has a zero in the interval $(a_1,a_2)$.

For contradiction we assume $\fss>0$ on $(a_1,a_2)$ (the case $\fss<0$ can be treated quite similarly). In what follows, we will construct a strictly admissible function $\tf$ such that
\begin{equation*}
 \int_{0}^{1}|\tf(x)|\ud x< \int_{0}^{1}|\fss(x)|\ud x,
\end{equation*}
which clearly contradicts the definition of $\fss$.

Naturally, we need to take special care  in constructing $\tf$ if either $a_1\in\Gamma_2$ or $a_2\in\Gamma_2$ which was defined in Definition~\ref{def:strictadm}. Moreover, if we manage to preserve the height of any existing jump in any other case then (vi$^{\prime}$.b) is automatically fulfilled for $\tf$.

First of all, we notice that $\fss$ cannot have a bend at, say, $y\in(a_1,a_2)$ such that the slope before the bend is greater than  afterwards. We say $\fss$ has a \emph{bend} in $y$ if $\fss$ is continuous in $y$ and if it changes its slope in $y$. Indeed, let $\delta>0$ such that the slope of $\fss$ is constant on $[y-\delta,y)$ as well as on $(y,y+\delta]$. Then, as can be seen in Figure~\ref{fig:property1_wrongbend}, we may interchange those pieces such that the resulting function $\tf$ (solid) remains continuous in $[y-\delta,y+\delta]$. Its absolute integral, however, is smaller than that of $\fss$ (dashed).
\begin{figure}[!htbp]
 \centering
 \includegraphics[width=0.5\textwidth]{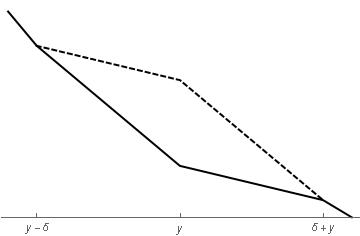}
 \caption{}
 \label{fig:property1_wrongbend}
\end{figure}
Thus, we need only consider bends where $\fss$ becomes flatter.

Let now $a_2\notin\Gamma_2$. We choose $\delta_1>0$ such that the slope of $\fss$ is a constant $s_1$ on $(a_2,a_2+\delta_1)$. Furthermore, we set
\begin{equation*}
 s=\min\left\{s^{*}(x):~x\in(a_1,a_2+\delta_1)\right\},
\end{equation*}
where $s^{*}$ denotes the slope of $\fss$ and where we define $s^{*}(a_2)$ as its left limit. Now, let $0<\delta\leq\min\{-2\fss(a_2)/(s_1+s),\delta_1\}$. With this choice of $\delta$ we have
\begin{equation*}
 \fss(a_2)+s\delta>-\fss(a_2+\delta).
\end{equation*}
In this case we may thus construct $\tf$ by  moving the discontinuity to $\tilde{a}_2=a_2+\delta$.  The missing part of $\tf$ on the left of $\tilde{a}_2$ of length $\delta$ is then chosen such that $\tf$ is continuous in $a_2$ and such that it has constant slope $s$.  This construction is visualized in Figure~\ref{fig:property1_notgamma2} (again $\fss$ is represented by the dashed and $\tf$ by the solid line). This choice for the slope guarantees that the height of the jump is preserved and, additionally, property~(vi$^{\prime}$.c) from Definition~\ref{def:strictadm}, too, cannot be violated by this construction if $a_1\in\Gamma_2$. 

\begin{figure}[!htbp]
 \centering
 \includegraphics[width=0.5\textwidth]{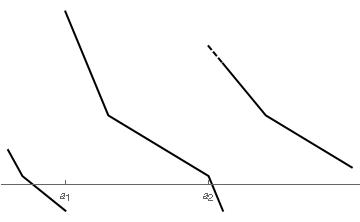}
 \caption{}
 \label{fig:property1_notgamma2}
\end{figure}

Certainly, the same construction also works if $a_2=\xi_{k_n}\in\Gamma_2$ for a suitable $k_{n}$ with $s^{*}\leq-a^{t-1}(a-2)-n$  between $a_2$ and the next discontinuity of $\fss$.

However, if there is some point $x>a_2$ before the next jump of $\fss$ with $s^{*}(x)>-a^{t-1}(a-2)-n$ we have to proceed differently. In this case, we keep the discontinuity at $a_2$ and take the smallest such  $x$, call it $\ox$. We define
\begin{equation*}
 \tf(x):=\left\{
 \begin{array}{{rl}}
  s_0(\ox-x)+\fss(\ox)&\quad \text{if }x\in[\ox-\delta,\ox),\\
  s^*(\ox)(\ox-\delta-x)+\tf(\ox-\delta) &\quad\text{if } x\in[a_2,\ox-\delta),\\
  \fss(x)&\quad\text{else, }
 \end{array}\right.
\end{equation*}
where $\delta>0$ is such that we still have a positive jump in $a_2$. Recall that a discontinuity always constitutes a positive jump, hence this is possible. Figure~\ref{fig:property1_gamma2} shows $\tf$ (solid) as well as $\fss$ (dashed) in this case. Notice that, again,
\begin{equation*}
 \int_{0}^{1}|\tf(x)|\ud x< \int_{0}^{1}|\fss(x)|\ud x
 \end{equation*}
 and that (vi$^{\prime}$.c) from Definition~\ref{def:strictadm} is not violated for $a_2$. Additionally, the condition on $\delta$ guarantees that (vi$^{\prime}$.c)  is not violated for $a_1$ if $a_1\in\Gamma_2$ either. Moreover, we need not take care of the height of the jump in $a_2$, since $\Gamma_1$ and $\Gamma_2$ are disjoint.  The dotted line represents the line with slope $s_0$ reaching back from $\{\ox,\fss(\ox)\}$ which occurs in Definition~\ref{def:strictadm}.
 \end{proof}

Thus, $\fss$ consists of parts $Q$, each of which is defined on an interval $[\alpha,\beta]$ with $\fss(\alpha)=\fss(\beta)=0$ and such that there is exactly one discontinuity in $(\alpha,\beta)$, see Figure~\ref{fig:q}.

\begin{figure}
\begin{minipage}[!hbtp]{0.45\textwidth}
\centering
 \includegraphics[width=\textwidth]{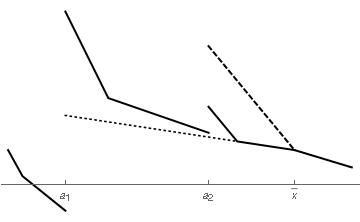}
 \caption{}
 \label{fig:property1_gamma2}
  \end{minipage}
  \hfill
  \begin{minipage}[!hbtp]{0.45\textwidth}
 \centering
 \includegraphics[width=\textwidth]{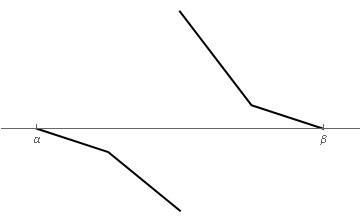}
 \caption{}
 \label{fig:q}
  \end{minipage}
 \end{figure}

In the following we determine the number of such $Q$'s for $\fss$.

\begin{Lemma}
 \label{lemma:numberjumps}
 The function $\fss$ has exactly $a^t-1$ discontinuities.
\end{Lemma}
\begin{proof}
 Assume that the total number of discontinuities of $\fss$ is less than $a^{t}-1$. Then, in the following, we will define a strictly admissible function $\tf$ from $\fss$ whose absolute integral is smaller than that of $\fss$. Let $\Gs$ be the set $\Gamma$ from property (vi$^{\prime}$) for the function $\fss$.
 
 By assumption there is a $\xis\in\Gs$ such that $\fss$ is continuous in $\xis$. The definition of $\Gs_1$ (i.e., the set $\Gamma_1$ for $\fss$) guarantees $\xis\notin\Gs_1$. Assume that $\xis\in\Gs_2$ (the case $\xis\in\Gs_0:=\Gs\setminus(\Gs_1\cup\Gs_2)$ can be treated quite analogously).
 
 Now choose $\gamma\in\Gs$ such that $\fss$ has a jump in $\gamma$. We show that $\gamma\in\Gs_1$ and that $\fss$ has a jump of height 1 in $\gamma$ (case d) below). Indeed, à priori we are in one of the following four cases:
 \begin{enumerate}[a)]
 \item $\gamma\in\Gs_0$,
 \item $\gamma\in\Gs_2$, 
 \item $\gamma\in\Gs_1$ with a jump of height greater than 1, or
 \item $\gamma\in\Gs_1$ with a jump of height exactly equal to 1 in $\gamma$.
\end{enumerate}
Assume that $\gamma\in\Gs_2$ (case b). By Lemma~\ref{lemma:zeros} $\gamma$ is isolated by two successive zeros of $\fss$. Hence (\ref{eqn:nr2}) from property (vi$^{\prime}$) cannot hold, and therefore (\ref{eqn:nr1}) from the same property does not hold either. Consequently, (see Fig.~\ref{fig:property2_gamma2}) we can take a point $\txi$ on the left of $\gamma$ and insert a short piece of minimal slope on $[\txi,\gamma)$ without interferring with property (vi$^{\prime}$.c). Again, the dashed line represents $\fss$ and the solid one the resulting new function $\tf$. The new set $\tG$ is the set $\Gs$ with $\xis$ replaced by $\txi$.

 \begin{figure}[!htbp]
 \centering
 \includegraphics[width=0.5\textwidth]{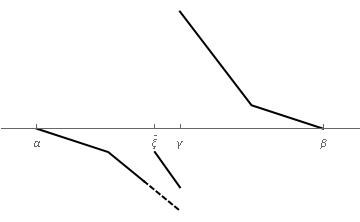}
 \caption{}
 \label{fig:property2_gamma2}
\end{figure}

This construction also works for case a) in the same way, and, with some special care, i.e., the jump of $\tf$ in $\gamma$ maintains a height of at least one, for case c) too.
 
 Consequently, $\fss$ can only have the $a^{t-1}(a-2)$ jumps at the positions given by $\Gs_1$. All these jumps have height exactly equal to one and there are absolutely no further jumps. Obviously, $\fss$ cannot have slope $-a^{t}$ everywhere, since then
 \begin{equation*}
  0>a^{t-1}(a-2)-a^{t}=\fss(1),
 \end{equation*}
 a contradiction to property (i). Thus, there exists an interval $[\delta_1,\delta_2]$ such that $\fss>0$ (or $\fss<0$) on $[\delta_1,\delta_2]$ and its slope is greater than $-a^{t}$. We choose $\delta^{\prime}\in(\delta_1,\delta_2)$ sufficiently close to $\delta_1$ (or to $\delta_2$) and define
 \begin{equation*}
 \tf(x)=\left\{
 \begin{array}{rl}
  \fss(\delta_1)-a^{t}(x-\delta_1)\quad&\text{if }x\in(\delta_1,\delta^{\prime}],\\
  \fss(x)\quad&\text{else,}
 \end{array}
\right. 
\end{equation*}
or
\begin{equation*}
 \tf(x)=\left\{
 \begin{array}{rl}
  \fss(\delta_2)-a^{t}(x-\delta_2)\quad&\text{if }x\in(\delta^{\prime},\delta_2],\\
  \fss(x)\quad&\text{else,}
 \end{array}
\right. 
\end{equation*}
respectively. See Figures~\ref{fig:property2_fspositive} and \ref{fig:property2_fsnegative}.

 \begin{figure}
\begin{minipage}[!hbtp]{0.45\textwidth}
 \centering
  \includegraphics[width=\textwidth]{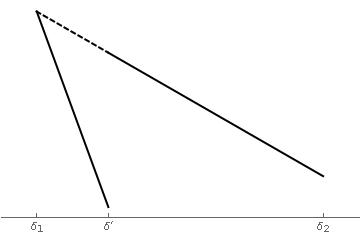}
  \caption{Case $\fss>0$ on $[\delta_1,\delta_2]$}
  \label{fig:property2_fspositive}
  \end{minipage}
  \hfill
  \begin{minipage}[!hbtp]{0.48\textwidth}
  \centering
  \includegraphics[width=\textwidth]{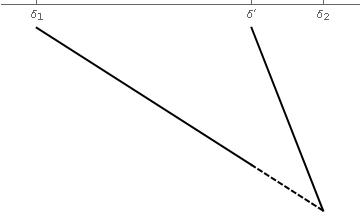}
  \caption{Case $\fss<0$ on $[\delta_1,\delta_2]$}
  \label{fig:property2_fsnegative}
  \end{minipage}
 \end{figure}
 \end{proof}
 
 From the above results we obtain that $\fss$ has to be of the following form: It divides $\I$ into $a^{t}-1$ parts $[\alpha,\beta)$ with $\fss(\alpha)=\fss(\beta)=0$, and $\fss$ has exactly one discontinuity $\gamma\in(\alpha,\beta)$. We say that $[\alpha,\beta)$ is of type $Q_i$ if $\gamma\in\Gs_i$ for $i=0,1,2$.
 
 From \cite{LarOnt14}, equation (2), we know that, for an interval of type $Q_0$ (this corresponds to the type $\Qpp$ in the abovementioned paper), we have
 \begin{equation*}
  \int_{\alpha}^{\beta}\left|\fss(x)\right|\ud x\geq\chi^2~\frac{a^{t-1}(a-2)}{4},\qquad \chi=\beta-\alpha,
 \end{equation*}
and from \cite[Lemma 2.12]{LarOnt14} and the considerations following the proof of this lemma we know that for an interval of type $Q_1$ (this corresponds to the type $\Qp$ in the abovementioned paper) we have
\begin{equation*}
 \int_{\alpha}^{\beta}\left|\fss(x)\right|\ud x\geq\frac{\chi\left(4-a^{t-1}\chi\right)}{16},\qquad \chi=\beta-\alpha.
\end{equation*}
Moreover, we know from \cite[Lemma 2.10]{LarOnt14} that for $\fss$ all $a^{t-1}$ intervals $Q_0$ have the same length and all $a^{t}-2a^{t-1}$ intervals $Q_1$ have the same length.

\begin{Lemma}
 \label{lemma:qpp}
 For $1\leq n\leq a^{t-1}-1$ let $Q_2^{(n)}$ be given by the interval $[\alpha,\beta)$. Then we have
 \begin{equation*}
  \int_{Q_2^{(n)}} \left|\fss(x)\right|\ud x\geq(\beta-\alpha)^{2}\frac{|s_0|(n+|s_0|)}{2(n+2|s_0|)}.
 \end{equation*}
\end{Lemma}
\begin{proof}
 This follows from the remark preceeding Lemma~\ref{lemma:qpp} and simple calculations.
\end{proof}
To finish the proof of our theorem we finally show:
\begin{Lemma}
 \label{lemma:lowerbound}
 For all $3\leq a\leq3.7$ we have
 \begin{equation*}
  \int_{0}^{1}\left|\fss(x)\right|\ud x \geq \frac{(a-2)\left(12a+9+(a-2)(4a-3)\log\left(1+\frac{1}{a-2}\right)\right)}{16\left(a-\frac{1}{2}\right)^2\left(3+(a-2)\log\left(1+\frac{1}{a-2}\right)\right)}.
 \end{equation*}
\end{Lemma}
\begin{proof}
 Due to Lemma~\ref{lemma:qpp} and the remarks preceeding it we have to minimize the right hand-side of
 \begin{multline*}
  \int_{0}^{1}\left|\fss(x)\right|\ud x \geq a^{t-1} \cdot\chi_0^2~\frac{a^{t-1}(a-2)}{4}+a^{t-1} (a-2)\cdot\frac{\chi_1\left(4-a^{t-1}\chi_1\right)}{16}
  +\\ \quad+
  \sum_{n=1}^{a^{t-1}-1}\left(\chi_{2}^{(n)}\right)^2\frac{|s_0|(n+|s_0|)}{2(n+2|s_0|)}\\
  =: a^{t-1} \cdot\chi_0^2 \tA_0+a^{t-1}(a-2) \cdot\frac{\chi_1\left(4-a^{t-1}\chi_1\right)}{16}+\sum_{n=1}^{a^{t-1}-1}\left(\chi_{2}^{(n)}\right)^2\tA_n
 \end{multline*}
with respect to $\chi_0,\chi_1,\chi_2^{(n)}\geq0$ (these quantities denote the lengths of the intervals $Q_0,Q_1,Q_2^{(n)}$) under the constraint
\begin{equation*}
 a^{t-1}\chi_0+a^{t-1}(a-2) \chi_1+ \sum_{n=1}^{a^{t-1}-1}\chi_2^{(n)}=1.
\end{equation*}
The Lagrangian approach immediately implies $\tA_0\chi_0=\tA_n\chi_2^{(n)}$ for all $1\leq n< a^{t-1}$. The constraint therefore yields 
\begin{equation*}
 \chi_0=\frac{1-a^{t-1}(a-2)\chi_1}{a^{t-1}+\sum_{n=1}^{a^{t-1}-1}\frac{\tA_0}{\tA_n}}.
\end{equation*}
Moreover, the denominator in the above equation simplifies to
\begin{multline*}
 a^{t-1}+\sum_{n=1}^{a^{t-1}-1} \frac{\tA_0}{\tA_n}=a^{t-1}+ \sum_{n=1}^{a^{t-1}-1}\left( 1-\frac{n}{2(|s_0|+n)} \right)
 =\\= 2a^{t-1}-1-\frac{1}{2} \sum_{n=|s_0|+1}^{a^{t-1}-1+|s_0|}\left(1-\frac{|s_0|}{n}\right)
  = \frac{1}{2}\left(3a^{t-1}-1+|s_0|  \sum_{n=|s_0|+1}^{a^{t-1}-1+|s_0|}\frac{1}{n} \right).
\end{multline*}
The latter sum  can be bounded by $\log(1+1/(a-2))$ from above. We summarize our intermediate findings and obtain
\begin{multline*}
 \int_{0}^{1}|\fss(x)|\ud x\geq\\
 \frac{(a-2)\left( 1-a^{t-1}(a-2)\chi_1 \right)^2}{2\left(3+(a-2)\log\left(1+\frac{1}{a-2}\right)\right)}+a^{t-1}(a-2)\frac{\chi_1(4-a^{t-1}\chi_1)}{16} =:p(\chi_1).
\end{multline*}
Now, our goal is to minimize the function $p$. We immediately see that $p$ is a polynomial of degree two and its leading coefficient is positive for all $3<a\leq3.7$. Thus, it attains its minimum at its only critical point
\begin{equation*}
 \cc=a^{1-t}\frac{2 \left( 4a-11-(a-2)\log\left(1+\frac{1}{a-2}\right) \right)}{29+8a(a-4)-(a-2)\log\left(1+\frac{1}{a-2}\right)}.
\end{equation*}
On the other hand, from the proof of Lemma 2.13 in \cite{LarOnt14} we know that we have the following bounds for $\chi_1$
\begin{equation*}
 \cm:=\frac{a^{1-t}}{a-\frac{1}{2}}\leq\chi_1\leq \frac{a^{1-t}}{a-\frac{3}{2}}.
\end{equation*}
We will show that $\cc\leq\cm$. Indeed, it can easily be  verified that the denominator of $\cc$ is positive.  Thus, $\cc>\cm$ iff
\begin{equation*}
 0 > 
 3a-9-(a-1)(a-2)\log\left(1+\frac{1}{a-2}\right)
 =:q(a).
\end{equation*}
We observe that $q(3.7)<0$ and, additionally, that $q'(a)>0$ for all $a\in(3,3.7]$. Hence $\chi_1=\frac{a^{1-t}}{a-\frac{1}{2}}$ and by inserting this value into the function $p$ the result follows.
\end{proof}

\bibliography{mybib}
\bibliographystyle{plain}
 \end{document}